\documentclass[12pt,reqno]{amsart}
\usepackage{graphicx}
\usepackage{amssymb,amsmath}
\usepackage{amsthm}
\usepackage{color}
\usepackage[pdf]{pstricks}
\usepackage{hyperref}
\usepackage{empheq}
\usepackage{pgfplots}

\usepackage{amssymb}
\usepackage{epsfig}
\usepackage{extarrows}
\usepackage{amsfonts}
\usepackage{graphicx}
\usepackage{caption}
\usepackage{subcaption}
\usepackage{tikz}
\usepackage{color}

\newcommand{\R}{\mathbb{R}}

\newtheorem{theorem}{Theorem}
\newtheorem{remark}{Remark}

\newtheorem{lemma}{Lemma}

\begin{document}
	
\title[Peaked Stokes waves]{\bf Peaked Stokes waves as solutions of Babenko's equation}

\author{Spencer Locke}
\address[S. Locke]{Department of Mathematics and Statistics, McMaster University, Hamilton, Ontario, Canada, L8S 4K1}
	
\author{Dmitry E. Pelinovsky}
\address[D. E. Pelinovsky]{Department of Mathematics and Statistics, McMaster University, Hamilton, Ontario, Canada, L8S 4K1}
\email{dmpeli@math.mcmaster.ca}

\begin{abstract}
Babenko's equation describes traveling water waves in holomorphic coordinates. It has been used in the past to obtain properties of Stokes waves with smooth profiles analytically and numerically. We show in the deep-water limit that properties of Stokes waves with peaked profiles can also be recovered from the same Babenko's equation. In order to develop the local analysis of singularities, we rewrite Babenko's equation as a fixed-point problem near the maximal elevation level. As a by-product, our results rule out a corner point singularity in the holomorphic coordinates, which has been obtained in a local version of Babenko's equation. 
\end{abstract}

\maketitle

Since the work of G. Stokes \cite{Stokes}, the periodic traveling wave solutions to Euler's equations -- referred to as Stokes' waves --  have been studied in the context of incompressible and irrotational two-dimensional surface water waves. According to the famous Stokes conjecture \cite{Toland}, the peaked profiles of the traveling waves with wave speed $c$ form an angle of $120^o$ in physical coordinates $(x,y)$ at the spatially dependent surface $y = \eta(x-ct)$. Since the two-dimensional hydrodynamics can be analyzed by using the complex variable $z := x + iy$ and conformal transformations to the new complex variable $w := u + iv$, the study of the traveling waves with peaked profiles has relied for a long time on pseudo--differential equations rewritten in holomorphic coordinates \cite{T02}.  The fluid domain in the $(x,y)$-plane is transformed to the flat domain in the $(u,v)$-plane by using the holomorphic coordinates, with the surface elevation profile $\eta$ being a function of $u-ct$ at $v = 0$.

With the use of the holomorphic coordinates, the existence of the peaked Stokes waves has been proven in \cite{Amick} and the $120^o$ angle was proven in \cite{Plotnikov}. Convexity of Stokes waves was further established in \cite{PlotToland}. Asymptotic results about the limiting singularity 
of $\eta$ as a function of $u-ct$ were found in \cite{Grant,Tanvir} and were recently reviewed in \cite{Lushnikov} as the limit of smooth Stokes waves. Numerical approximations of the Stokes waves with smooth profiles break down before the wave can approach to the limiting peaked profile \cite{DCD16,Lush1,Lush2,Lush3}. Holomorphic coordinates were also used for analysis of well-posedness of Euler's water wave equations \cite{Ifrim1,Ifrim2}. Further studies on the peaked waves were undertaken for fluids with finite or infinite depth and under the constant vorticity \cite{KL21,KL23,V09,VW12}, as well as for periodic and solitary wave profiles in the model equations \cite{BD19,D23,EMV23,GP19}.

Without loss of generality, we consider the surface elevation $\eta$ on the $2\pi$-periodic domain denoted by $\mathbb{T}$ and assume that it satisfies 
the zero mean constraint $\oint \eta dx = 0$ in the physical coordinate $x$. 
Bernoulli's principle implies that the velocity potential satisfies 
\begin{equation}
\label{pressure}
-c \varphi_x + \frac{1}{2} (\varphi_x)^2 + \frac{1}{2} (\varphi_y)^2 + \eta = 0, \quad \mbox{\rm at} \;\; y = \eta(x-ct),
\end{equation}
where the pressure is set to $0$ at the free surface and the gravitational acceleration constant $g$ is set to $1$. It follows from (\ref{pressure}) with $\varphi_x = c$ and $\varphi_y = 0$ at the maximal height of the surface 
elevation for the peaked wave that the peaked profile satisfies
\begin{equation}
\label{max-height}
\max_{x \in \mathbb{T}} \eta(x) = \frac{c^2}{2}.
\end{equation}
The location of the maximal height can be placed at the origin due to the translational invariance. The existence and the singular behavior of the peaked wave with the profile $\eta$ in the holomorphic coordinate $u$ is well-known, see \cite{Lushnikov,Plotnikov}, and it satisfies the expansion
\begin{equation}
\label{sing-behavior}
\eta(u) = \frac{c^2}{2} - A |u|^{2/3} + {\rm o}(|u|^{2/3}) \quad \mbox{\rm as} \;\; u \to 0,
\end{equation}
where $A > 0$ is some constant and the admissible values of $c$ are not known explicitly. The singular behavior (\ref{sing-behavior}) leads to the $120^o$ angle in the physical coordinates $(x,y)$ \cite{Plotnikov}.

The main motivation for our work is to explore a closed equation for traveling Stokes waves in holomorphic coordinates called Babenko's equation, named after 
the pioneering work in \cite{Babenko}. Babenko's equation can be written for the surface elevation in the form 
\begin{equation}
\label{Babenko}
(c^2 K_h - 1) \eta - \eta K_h \eta - \frac{1}{2} K_h \eta^2 = 0,
\end{equation}
where $K_h$ is a linear, self-adjoint, positive operator in $L^2(\mathbb{T})$ defined by the Fourier symbol
\begin{equation}
\label{operator-K}
\widehat{\left( K_h \right)}_n = \left\{ \begin{array}{cl} n \coth(hn), &\quad n \in \mathbb{Z} \backslash \{0\},  \\ 0, & \quad n = 0, \end{array} \right.
\end{equation}
and $h$ is the fluid depth in the holomorphic coordinate $v$, see Appendix A in \cite{LP24}. Babenko's equation in the form (\ref{Babenko}) has been derived in the case of the zero-mean constraint  $\oint \eta dx = 0$ in the physical coordinate $x$, which is equivalent to the following constraint for the profile $\eta \in {\rm Dom}(K_h) \subset L^2(\mathbb{T})$, 
\begin{equation}
\label{mean-zero}
\oint \eta (1 + K_h \eta) du = 0,
\end{equation}
in the holomorphic coordinate $u$. Babenko's equation (\ref{Babenko}) was used in \cite{DZ1,DZ2} to get small-amplitude expansions of Stokes waves with smooth profiles and the conserved quantities in the dynamics of water waves. It was used for numerical approximations of Stokes waves with steeper, though still smooth, profiles in \cite{Lush1,Lush2,Lush3}.  

{\em The purpose of this paper is to show that Babenko's equation (\ref{Babenko}) gives accurate predictions for Stokes waves with peaked profiles, which agrees with the expansion (\ref{sing-behavior}) obtained directly from Euler's equations.} 

To make computations simple and short, we consider the deep-water limit ($h \to \infty$), for which Babenko's equation (\ref{Babenko}) is now used with $K_{\infty} := \lim\limits_{h \to \infty} K_h = -\mathcal{H} \partial_u$, where $\mathcal{H}$ is the periodic Hilbert transform defined by either the Fourier symbol 
\begin{equation}
\label{Hilbert}
\widehat{\left(\mathcal{H}\right)}_n = \left\{ \begin{array}{cl} i \; {\rm sgn}(n), &\quad n \in \mathbb{Z} \backslash \{0\},  \\ 0, & \quad n = 0, \end{array} \right.
\end{equation}
or the integral formulation 
\begin{equation}
\label{Hilbert-integral}
\mathcal{H}(f)(u) := \frac{1}{\pi} \sum_{n=-\infty}^{\infty} {\rm p.v} \int_{\mathbb{T}} \frac{f(u')}{u' - u + 2 \pi n} du'.
\end{equation}
The integral formulation (\ref{Hilbert-integral}) is particularly useful to prove the following main results. 

\begin{theorem}
	\label{theorem-1}
	Assume that Babenko's equation (\ref{Babenko}) in the deep-water limit $h \to \infty$ admits a solution with the local behavior 
\begin{equation}
\label{behavior-1}
\eta(u) = \frac{c^2}{2} - A |u|^{\beta} + {\rm o}(|u|^{\beta}) \quad \mbox{\rm as} \;\; u \to 0,
\end{equation}	
where $\beta \in (0,1]$ for some admissible values of $A > 0$ and $c \in \mathbb{R}$. Then, $\beta = \frac{2}{3}$. 
\end{theorem}

\begin{theorem}
	\label{theorem-2}
	Assume that Babenko's equation (\ref{Babenko}) in the deep-water limit $h \to \infty$ admits a solution with the local behavior 
	\begin{equation}
	\label{behavior-2}
	\eta(u) = \frac{c^2}{2} - A |u|^{2/3} - B |u|^{\mu} + {\rm o}(|u|^{\mu}) \quad \mbox{\rm as} \;\; u \to 0,
	\end{equation}	
	where $\mu \in \left(\frac{2}{3},2\right)$ for some admissible values of $A > 0$, $B \in \R$, and $c \in \mathbb{R}$. Then, $\mu$ is the second root of the transcendental equation 
	\begin{equation}
	\label{Grant-eq}
	\left( \mu + \frac{2}{3} \right) \tan\left(\frac{\pi \mu}{2} + \frac{\pi}{3} \right) + \mu \tan \left( \frac{\pi \mu}{2}\right) + \frac{2}{\sqrt{3}} = 0.
	\end{equation}
\end{theorem}

\begin{remark}
	In our previous work \cite{LP24}, we considered Babenko's equation (\ref{Babenko}) with 
	$K_h \equiv -\partial_u^2$ as a toy model. We proved \cite[Theorem 1]{LP24} that there exist $c_* := \frac{\pi}{2 \sqrt{2}}$ and $c_{\infty} \in (c_*,\infty)$ such that 
the family of periodic solutions admits smooth profiles for $c \in (1,c_*)$ and 
singular profiles satisfying the expansion (\ref{sing-behavior}) for $c \in (c_*,c_{\infty})$. The periodic solution for $c = c_*$ has the peaked profile with $\eta \in C^0_{\rm per}(\mathbb{T}) \cap W^{1,\infty}(\mathbb{T})$; that is, it corresponds to the expansion (\ref{behavior-1}) with $\beta = 1$. Theorem \ref{theorem-1} rules out the possibility of such peaked profiles in the holomorphic coordinate $u$ at least in the deep-water limit. 
\end{remark}

\begin{remark}
	The transcendental equation (\ref{Grant-eq}) was derived in \cite{Grant} by direct analysis of Euler's equations and in a different form. The second root of this equation can be found numerically at $\mu \approx 1.469$. Hence the profile satisfies $\eta(u) + A |u|^{2/3} \in C^1_{\rm per}(\mathbb{T})$ beyond the leading-order singularity. 
	Thus, the known results for the asymptotic analysis of the traveling waves with the singular profiles can be recovered from Babenko's equation (\ref{Babenko}) in Theorem \ref{theorem-2}.
\end{remark}

\begin{remark}
	Local analysis in the proof of Theorems \ref{theorem-1} and \ref{theorem-2} does not allow us to obtain the admissible values for $A > 0$, $B \in \R$, and $c \in \mathbb{R}$. It is not clear if the admissible values of the wave speed $c$ includes just one point, as believed from numerical data in \cite{Lush1,Lush3}, or an interval of wave speeds, as follows from the local model in \cite{LP24}. This unsatisfactory conclusion is because these coefficients are obtained from the remainder terms of the asymptotic expansions, which cannot be deduced without a complete solution of Babenko's equation (\ref{Babenko}). 
\end{remark}

\begin{remark}
	Since the existence of the Stokes wave with the peaked profile satisfying the expansion (\ref{sing-behavior}) has been proven in \cite{Amick,Plotnikov,PlotToland} and the singular behavior of Theorems \ref{theorem-1} and \ref{theorem-2} for Babenko's equation (\ref{Babenko}) agrees with (\ref{sing-behavior}), we strongly believe that the existence assumption is satisfied in Babenko's equation. The proof of existence is beyond the scope of this work.
\end{remark}

To prove the main results of Theorems \ref{theorem-1} and \ref{theorem-2}, we use (\ref{max-height}) and introduce the deviation of the surface elevation from the maximal height as 
\begin{equation}
\label{tilde-eta}
\tilde{\eta} := \frac{c^2}{2} - \eta.
\end{equation} 
Using (\ref{Babenko}) with $K_{\infty} = - \mathcal{H} \partial_u$, we obtain the following fixed-point equation for the deviation coordinate
\begin{equation}
\label{fixed-point}
\tilde{\eta} = T_{\infty} \tilde{\eta} := \frac{c^2}{2} + \frac{1}{2} K_{\infty} \tilde{\eta}^2 + \tilde{\eta} K_{\infty} \tilde{\eta}.
\end{equation}
The method of the proof is to assume the expansions (\ref{behavior-1}) and (\ref{behavior-2}) of $\eta$ in fractional powers of $u$ and to get a contradiction with the range of $T_{\infty}$. We note the following table integral from \cite{GR} for every $\nu \in (0,1)$ and $u \in \mathbb{R}$:
\begin{equation}
\label{table-intergral}
\frac{1}{\pi} {\rm p.v.} \int_{-\infty}^{\infty} \frac{|u'|^{\nu - 1}}{u' - u}du' = 
- \cot \left( \frac{\pi \nu}{2} \right) |u|^{\nu - 1} {\rm sgn}(u). 
\end{equation}
Based on the exact formula (\ref{table-intergral}) and the decomposition of (\ref{Hilbert-integral}), we show the following.

\begin{lemma}
	\label{lem-sing}
	For every $\nu \in (0,1)$, $u_0 \in (0,\pi)$, and $u \in (-u_0,u_0)$ we have
\begin{equation}
\label{Hilbert-table}
\mathcal{H}(|u|^{\nu - 1}) = 
	- \cot \left( \frac{\pi \nu}{2} \right) |u|^{\nu - 1} {\rm sgn}(u) + F_{\nu}(u), 
\end{equation}
where $F_{\nu} \in C^{\infty}(-u_0,u_0)$.
\end{lemma}

\begin{proof}
	We can write 
\begin{align*}
{\rm p.v.} \int_{-\infty}^{\infty} \frac{|u'|^{\nu - 1}}{u' - u}du' = {\rm p.v.} \int_{-\pi}^{\pi} \frac{|u'|^{\nu - 1}}{u' - u}du' + \left( \int_{-\infty}^{-\pi} + \int^{\infty}_{\pi} \right) \frac{|u'|^{\nu - 1}}{u' - u}du',
\end{align*}
where the remainder term is $C^{\infty}$ for any $u \in (-\pi,\pi)$ if $\nu \in (0,1)$. 
On the other hand, writing
\begin{align*}
\pi \mathcal{H}(|u|^{\nu - 1})  &= {\rm p.v.} \int_{-\pi}^{\pi} \frac{|u'|^{\nu - 1}}{u' - u}du' + \sum_{n\in \mathbb{Z} \backslash \{0\}} \int_{-\pi}^{\pi} \frac{|u'|^{\nu-1}}{u'-u+ 2\pi n}du'
\end{align*}
and estimating the remainder term by using Cauchy estimates (see Theorem 7.1 in \cite{Locke-thesis}) shows that 
the remainder term is holomorphic on any compact subset of $(-\pi,\pi)$. Combining the two expansions together with the table integral (\ref{table-intergral}) yields (\ref{Hilbert-table}).
\end{proof}

Since $\mathcal{H}^2 = -{\rm Id}$, we also obtain from Lemma \ref{lem-sing} with a similar decomposition that 
\begin{lemma}
	\label{lem-sing-2}
	For every $\nu \in (0,1)$, $u_0 \in (0,\pi)$, and $u \in (-u_0,u_0)$ we have
	\begin{equation}
	\label{Hilbert-table-2}
	\mathcal{H}(|u|^{\nu - 1} {\rm sgn}(u)) = 
\tan \left( \frac{\pi \nu}{2} \right) |u|^{\nu - 1} + F_{\nu}(u), 
	\end{equation}
	where $F_{\nu} \in C^{\infty}(-u_0,u_0)$.
\end{lemma}

We are now ready to prove Theorems \ref{theorem-1} and \ref{theorem-2}.

\begin{proof}[Proof of Theorem \ref{theorem-1}]
	By using (\ref{behavior-1}) and (\ref{tilde-eta}), we assume 
	$\tilde{\eta}(u) = A |u|^{\beta} + {\rm o}(|u|^{\beta})$ for some $A > 0$ and $\beta \in (0,1]$. By Lemma \ref{lem-sing-2}, we get for every $\beta \in (0,1)$ that
	\begin{align}
	\label{exp-1}
	K_{\infty} |u|^{\beta} = - \beta \mathcal{H}(|u|^{\beta-1}\text{sgn}(u)) = -\beta\tan\left(\frac{\pi\beta}{2}\right)|u|^{\beta-1} + \mathcal{O}(1) \quad \mbox{\rm as} \;\; |u| \to 0.
	\end{align}
For $K_{\infty} |u|^{2\beta}$, we have to consider separately $2 \beta \in (0,1)$ and $2 \beta \in (1,2)$. If $2\beta\in (0,1)$, we get 
	\begin{align}
	\label{exp-2}
	K_{\infty} |u|^{2\beta} = - 2\beta \mathcal{H} (|u|^{2\beta-1}\text{sgn}(u)) =-2\beta\tan(\pi\beta)|u|^{2\beta-1} + \mathcal{O}(1) \quad \mbox{\rm as} \;\; |u| \to 0.
	\end{align}
This yields 
	\begin{align}
	\label{exp-3}
T_{\infty} \tilde{\eta} = -A^2 \beta \left[ \tan\left(\pi\beta\right)+\tan\left(\frac{\pi\beta}{2}\right)\right] |u|^{2\beta-1} + \mathcal{O}(1) + {\rm o}(|u|^{2\beta -1}) \quad \mbox{\rm as} \;\; |u| \to 0.
	\end{align}
Since $2 \beta - 1 < \beta$, the fixed-point equation $\tilde{\eta} = T_{\infty} \tilde{\eta}$ can be satified as $|u| \to 0$ if and only if 
	\begin{align}
\label{exp-4}
	\tan\left(\pi\beta\right)+\tan\left(\frac{\pi\beta}{2}\right)=0,
\end{align}
which is equivalent to $1 + 2 \cos(\pi \beta) = 0$ for $2 \beta \in (0,1)$. 
However, there are no roots of this equation for $2 \beta \in (0,1)$.

If $2 \beta \in (1,2)$, the expansion (\ref{exp-2}) is not valid. By Lemma \ref{lem-sing}, we compute 
	\begin{align*}
	\frac{\partial}{\partial u} K_{\infty} |u|^{2\beta} 
	&= -2 \beta (2 \beta - 1) \mathcal{H} (|u|^{2\beta - 2}) \\
	& = 2 \beta (2\beta-1) \cot\left(\frac{\left(2\beta-1\right)\pi}{2}\right)|u|^{2\beta-2}\text{sgn}(u) + \mathcal{O}(1) \quad \mbox{\rm as} \;\; |u| \to 0.
	\end{align*}
Integrating this equation yields
	\begin{align}
\label{exp-5}
K_{\infty} |u|^{2\beta}=2\beta\cot\left(\frac{\left(2\beta-1\right)\pi}{2}\right)|u|^{2\beta-1} + \mathcal{O}(1) \quad \mbox{\rm as} \;\; |u| \to 0.
\end{align}
Since $\cot\left(\frac{\left(2\beta-1\right)\pi}{2}\right) = -\tan(\pi \beta)$, 
expansions (\ref{exp-1}) and (\ref{exp-5}) in (\ref{fixed-point}) yields the same expansion (\ref{exp-3}) for $2 \beta \in (1,2)$. Since $2 \beta - 1 < \beta$, the balance of singular terms in $\tilde{\eta} = T_{\infty} \tilde{\eta}$ as $|u| \to 0$ is satisfied under the same condition (\ref{exp-4}), which is equivalent to $1 + 2 \cos(\pi \beta) = 0$ for $2 \beta \in (1,2)$. Hence, there exists a unique root $\beta = \frac{2}{3}$ in the admissible range $2 \beta \in (1,2)$.

It remains to rule out the marginal cases $2 \beta = 1$ and $\beta = 1$. 
It follows from Cauchy estimates in \cite[Theorem 7.1]{Locke-thesis} and 
the explicit computations of (\ref{Hilbert-integral})  that  
\begin{align*} 
\mathcal{H} (|u|) = -\frac{2}{\pi} u \ln |u| + \mathcal{O}(u) \quad \mbox{\rm as} \;\; |u| \to 0,
\end{align*}
which yields 
\begin{align}
\label{exp-6} 
K_{\infty} |u| = \frac{2}{\pi} \ln |u| + \mathcal{O}(1) \quad \mbox{\rm as} \;\; |u| \to 0.
\end{align}

If $2 \beta = 1$, then $\tilde{\eta} = \mathcal{O}(|u|^{1/2})$. Thus, expansion (\ref{exp-1}) holds and gives $\tilde{\eta} K_\infty \tilde{\eta} = \mathcal{O}(1)$, whereas (\ref{exp-6}) shows that we will have that $K_\infty \tilde{\eta}^2 = \mathcal{O(}\ln{|u|})$. Hence, it follows by (\ref{fixed-point}) that $T_\infty \tilde{\eta} = \mathcal{O}(\ln{|u|}) + \mathcal{O}(1)$. Since there is no balance of the singular terms $\mathcal{O(}\ln{|u|})$ in  $\tilde{\eta} = T_{\infty} \tilde{\eta}$ as $|u| \to 0$, the case $2 \beta = 1$ is impossible. 

If $\beta = 1$, then $\tilde{\eta} = \mathcal{O}(|u|)$ and so (\ref{exp-6}) implies that $\tilde{\eta} K_\infty \tilde{\eta} =  \mathcal{O}(|u|\ln|u|)$, while on the other hand one observes that $K_{\infty} |u|^2 = K_{\infty} u^2$, which is smooth as $|u| \to 0$. Hence, it follows by (\ref{fixed-point}) that $T_\infty \tilde{\eta} = \mathcal{O}(|u| \ln{|u|}) + \mathcal{O}(1)$. Since there is no balance of the singular terms $\mathcal{O}(|u| \ln|u|)$ in  $\tilde{\eta} = T_{\infty} \tilde{\eta}$ as $|u| \to 0$, then the case $\beta = 1$ is also impossible. 
\end{proof}

\begin{proof}[Proof of Theorem \ref{theorem-2}]
	By using (\ref{behavior-2}) and (\ref{tilde-eta}), we assume 
$\tilde{\eta}(u) = A |u|^{2/3} + B |u|^{\mu} + {\rm o}(|u|^{\mu})$ for some $A > 0$, $B \in \R$, and $\mu \in \left(\frac{2}{3},2\right)$. Similarly to (\ref{exp-1}), (\ref{exp-2}), and (\ref{exp-5}), we obtain for $\mu \neq 1$ 
that 
\begin{align*}
\tilde{\eta} K_{\infty} \tilde{\eta} = 
-\frac{2}{\sqrt{3}} A^2 |u|^{1/3} - AB \left[\frac{2}{\sqrt{3}} + \mu \tan \left(\frac{\mu \pi}{2} \right) \right] |u|^{\mu - 1/3} + \mathcal{O}(|u|^{2/3}) + {\rm o}(|u|^{\mu-1/3})
\end{align*}
and 
\begin{align*}
\frac{1}{2} K_{\infty} \tilde{\eta}^2 = 
\frac{2}{\sqrt{3}} A^2 |u|^{1/3} - AB \left(\mu + \frac{2}{3}\right) \tan \left(\frac{\pi \mu}{2} + \frac{\pi}{3} \right) |u|^{\mu - 1/3} + \mathcal{O}(1) + {\rm o}(|u|^{\mu-1/3})
\end{align*}
where we took into account that $2 \mu - 1 > \mu - \frac{1}{3}$ for $\mu > \frac{2}{3}$. When these expansions are substituted into the fixed-point equation (\ref{fixed-point}), the terms of $\mathcal{O}(|u|^{2/3})$ and $\mathcal{O}(1)$ determine a nonlocal problem for $A$ and $c$, whereas the singular term $\mathcal{O}(|u|^{\mu - 1/3})$ for $\mu \neq 1$ is not matched and must vanish. This yields the transcendental equation (\ref{Grant-eq}). The first root  $\mu = \frac{2}{3}$ of (\ref{Grant-eq}) is excluded for $\mu \in \left(\frac{2}{3},1\right)$. The second root  of (\ref{Grant-eq}) can be found graphically at $\mu \approx 1.469$ \cite{Grant}.

For the marginal case $\mu = 1$, $\tilde{\eta} K_{\infty} \tilde{\eta}$ produces $\mathcal{O}(|u|^{2/3} \ln|u|)$ due to expansion (\ref{exp-6}) and $K_{\infty} \tilde{\eta}^2$ produces $\mathcal{O}(|u|^{2/3})$ as $|u| \to 0$. Since there is no balance of singular terms in $\tilde{\eta} = T_{\infty} \tilde{\eta}$  as $|u| \to 0$, the marginal case $\mu = 1$ is excluded.
\end{proof}



\end{document}